\documentclass[12pt]{article}
\usepackage{amsmath,amssymb,amsbsy,amsfonts,amsthm,latexsym,amsopn,amstext,
                            amsxtra,euscript,amscd}
\usepackage{booktabs}

\newfont{\teneufm}{eufm10}
\newfont{\seveneufm}{eufm7}
\newfont{\fiveeufm}{eufm5}
\newfam\eufmfam
                 \textfont\eufmfam=\teneufm \scriptfont\eufmfam=\seveneufm
                 \scriptscriptfont\eufmfam=\fiveeufm
\def\bbbc{{\mathchoice {\setbox0=\hbox{$\displaystyle\rm C$}\hbox{\hbox
to0pt{\kern0.4\wd0\vrule height0.9\ht0\hss}\box0}}
{\setbox0=\hbox{$\textstyle\rm C$}\hbox{\hbox
to0pt{\kern0.4\wd0\vrule height0.9\ht0\hss}\box0}}
{\setbox0=\hbox{$\scriptstyle\rm C$}\hbox{\hbox
to0pt{\kern0.4\wd0\vrule height0.9\ht0\hss}\box0}}
{\setbox0=\hbox{$\scriptscriptstyle\rm C$}\hbox{\hbox
to0pt{\kern0.4\wd0\vrule height0.9\ht0\hss}\box0}}}}
\def\bbbq{{\mathchoice {\setbox0=\hbox{$\displaystyle\rm
Q$}\hbox{\raise
0.15\ht0\hbox to0pt{\kern0.4\wd0\vrule height0.8\ht0\hss}\box0}}
{\setbox0=\hbox{$\textstyle\rm Q$}\hbox{\raise
0.15\ht0\hbox to0pt{\kern0.4\wd0\vrule height0.8\ht0\hss}\box0}}
{\setbox0=\hbox{$\scriptstyle\rm Q$}\hbox{\raise
0.15\ht0\hbox to0pt{\kern0.4\wd0\vrule height0.7\ht0\hss}\box0}}
{\setbox0=\hbox{$\scriptscriptstyle\rm Q$}\hbox{\raise
0.15\ht0\hbox to0pt{\kern0.4\wd0\vrule height0.7\ht0\hss}\box0}}}}
\def\bbbt{{\mathchoice {\setbox0=\hbox{$\displaystyle\rm
T$}\hbox{\hbox to0pt{\kern0.3\wd0\vrule height0.9\ht0\hss}\box0}}
{\setbox0=\hbox{$\textstyle\rm T$}\hbox{\hbox
to0pt{\kern0.3\wd0\vrule height0.9\ht0\hss}\box0}}
{\setbox0=\hbox{$\scriptstyle\rm T$}\hbox{\hbox
to0pt{\kern0.3\wd0\vrule height0.9\ht0\hss}\box0}}
{\setbox0=\hbox{$\scriptscriptstyle\rm T$}\hbox{\hbox
to0pt{\kern0.3\wd0\vrule height0.9\ht0\hss}\box0}}}}
\def\bbbs{{\mathchoice
{\setbox0=\hbox{$\displaystyle     \rm S$}\hbox{\raise0.5\ht0\hbox
to0pt{\kern0.35\wd0\vrule height0.45\ht0\hss}\hbox
to0pt{\kern0.55\wd0\vrule height0.5\ht0\hss}\box0}}
{\setbox0=\hbox{$\textstyle        \rm S$}\hbox{\raise0.5\ht0\hbox
to0pt{\kern0.35\wd0\vrule height0.45\ht0\hss}\hbox
to0pt{\kern0.55\wd0\vrule height0.5\ht0\hss}\box0}}
{\setbox0=\hbox{$\scriptstyle      \rm S$}\hbox{\raise0.5\ht0\hbox
to0pt{\kern0.35\wd0\vrule height0.45\ht0\hss}\raise0.05\ht0\hbox
to0pt{\kern0.5\wd0\vrule height0.45\ht0\hss}\box0}}
{\setbox0=\hbox{$\scriptscriptstyle\rm S$}\hbox{\raise0.5\ht0\hbox
to0pt{\kern0.4\wd0\vrule height0.45\ht0\hss}\raise0.05\ht0\hbox
to0pt{\kern0.55\wd0\vrule height0.45\ht0\hss}\box0}}}}
\def\bbbz{{\mathchoice {\hbox{$\sf\textstyle Z\kern-0.4em Z$}}
{\hbox{$\sf\textstyle Z\kern-0.4em Z$}}
{\hbox{$\sf\scriptstyle Z\kern-0.3em Z$}}
{\hbox{$\sf\scriptscriptstyle Z\kern-0.2em Z$}}}}

\newtheorem{theorem}{Theorem}

\newtheorem{lem}[theorem]{Lemma}
\def\cF{{\mathcal F}}

\def\cL{{\mathcal L}}

\def\cQ{{\mathcal Q}}
\def\cR{{\mathcal R}}
\def\cS{{\mathcal S}}

\def\({\left(}
\def\){\right)}
\def\[{\left[}
\def\]{\right]}
\def\<{\langle}
\def\>{\rangle}

\def\rf#1{\left\lceil#1\right\rceil}

\def\Z{\mathbb{Z}}
\def\Q{\mathbb{Q}}

\def\Sp {\cS_p}

\newtheorem{algorithm}[theorem]{Algorithm}

\def\mand{\qquad\mbox{and}\qquad}

\begin{document}

\title{\bf Statistics of Different Reduction Types of Fermat Curves}

\author{
{\sc   David Harvey} \\
{School of Mathematics and Statistics}\\ {University of New South Wales} \\
{Sydney, NSW 2052, Australia} \\
{\tt  d.harvey@unsw.edu.au} 
\and 
{\sc   Igor E.~Shparlinski} \\
{Department of Computing}\\ {Macquarie University} \\
{Sydney, NSW 2109, Australia} \\
{\tt  igor.shparlinski@mq.edu.au}}

\date{}

\maketitle

\begin{abstract} We present some theoretic bounds and algorithms   concerning the 
statistics  of different reduction types in the family of Fermat  curves
$Y^p = X^s(1-X)$, where $p$ is 
prime and  $s =1, \ldots, p-2$.
\end{abstract}


\section{Introduction}

For a prime $p$, we define $\Sp =\{1, \ldots, p-2\}$, and consider the family of curves
$$
\cF_s:\ Y^p = X^s(1-X), \qquad s \in \Sp,
$$
over the algebraic closure of $\Q$. 

It has been shown by McCallum~\cite{McC1} (see also~\cite{LeMcC,McC2,McCTz})
that there is a direct  link  between the properties of the reduction 
of $\cF_s$ modulo $p$ (in particular to tame, wild split and wild non-split
reductions) and the  {\it  Fermat quotients.\/}

We recall that for a prime $p$ and an integer $u$ with $\gcd(u,p) =1$
the Fermat quotient $q_p(u)$ is defined by the conditions:
$$
q_p(u) \equiv \frac{u^{p-1} -1}{p} \pmod p
\mand
0 \le q_p(u) < p.
$$
We define the sequence of Legendre symbols:
$$
\vartheta_{p,s} = \(\frac{2s(s+1) q_p\(s^s/(s+1)^{s+1}\)}{p}\),
 \qquad s \in \Sp.
$$
By a result of McCallum~\cite{McC1} (see also~\cite{LeMcC,McC2,McCTz}),
the curve $F_s$ is tame, wild split or wild non-split depending on whether 
$\vartheta_{p,s}=0$, $1$ or  $-1$, respectively.

For $\vartheta= 0,\pm 1$, we define $N_\vartheta(p)$ as the number of integers $s \in \Sp$ with $\vartheta_{p,s}=\vartheta$. A natural conjecture is that $2s(s+1) q_p\(s^s/(s+1)^{s+1}\)$ behaves uniformly randomly modulo $p$, and that the values for the various $s$ are independent. Therefore one expects $\vartheta_{p,s} = \pm 1$ both occur about half the time, and $\vartheta_{p,s} = 0$ occurs for about $1/p$ values of $s$ on average. In other words, one expects that $N_0(p)$ behaves like a Poisson random variable with mean $1$, and that $N_{\pm1}(p)$ behave like normal random variables with mean $p/2$ and variance $p/2$, as $p \to \infty$.

We show that these heuristics are not quite correct: there are \emph{relations} among the $\vartheta_{p,s}$, so that the independence hypothesis must be modified. After adjusting our heuristics to take into account these relations, we find excellent numerical agreement with a table of values of $N_{0, \pm 1}(p)$ for $p < 10^7$.

We also prove that $N_0(p) = O(p^{2/3})$. This confirms that $\vartheta_{p,s} = 0$ occurs relatively rarely. Unfortunately this bound is much weaker than the bounds suggested by our numerical experiments.

Clearly the question about the distribution 
of the values of $N_{\pm 1}(p)$ is essentially a 
question about bounding a certain sum of Legendre symbols 
with Fermat quotients. Recently there has been some
progress in the area of estimating multiplicative character sums
with Fermat quotients $q_p(u)$, see~\cite{Chang,Shp1,Shp2,Shp3}. 
However,  the results
and methods of these papers do not seem to apply to
sums corresponding to $\vartheta_{p,s}$. 

\section{Preparations}
\label{sec:prep}

Our bound of $N_0(p)$ is based on the following
bound of Heath-Brown~\cite[Lemma~2]{H-B} (see also~\cite{Mit}).
Let
$$
f(u) = \sum_{j=1}^{p-1}\frac{u^j}j.
$$

\begin{lem}\label{lem:HB}
For any integer $r$ the congruence
$$
f(u) \equiv r  \pmod p, \qquad 2 \le u \le p-1,
$$
has $O(p^{2/3})$ solutions. 
\end{lem}

We also recall some notions and the results from 
the  uniform distribution theory, see~\cite{DrTi}.
As usual, we define the {\it 
discrepancy\/} $\Delta(\Gamma)$ of a sequence 
$\Gamma= \(\gamma_{m}\)_{m=1}^{M}$ of $M$
(not necessarily distinct) points
in the  unit interval $[0,1]$ by
$$
\Delta(\Gamma) = \sup_{0 \le \gamma\le 1}
\left|\frac{I_\Gamma(\gamma)} {M} - \gamma\right|,
$$
where $I_\Gamma(\gamma)$ is the number of points $\gamma_m$ of  
$\Gamma$ with  $\gamma_m \le \gamma$. 

We trivially have:

\begin{lem}\label{lem:DiscrStable}
Let $\varPhi= \(\varphi_{m}\)_{m=1}^{M}$ and 
$\varPsi= \(\psi_{m}\)_{m=1}^{M}$ be two sequences of $M$
points
in the  unit interval $[0,1]$ with $|\varphi_{m}-\psi_{m}| \le \delta$, 
$m =1, \ldots, M$. Then for their discrepancies we have 
$\Delta(\varPhi) = \Delta(\varPsi) + O(\delta)$.
\end{lem}

The following result about the distribution of modular 
inverses is well-known, see~\cite[Section~3]{Shp4} for a
survey of several similar estimates. It has appeared in many works, and
follows instantly from the Weil bound for
Kloosterman sums, see~\cite[Theorem~11.11]{IwKow}.

\begin{lem}\label{lem:Invers}
For any prime $\ell$ and integers $0 < K < K+ R < \ell$ the discrepancy of the sequence 
$\overline r/\ell$, $r \in [K,K+R]$ is $O\(R^{-1} \ell^{1/2}(\log \ell)^{2}\)$,
where $\overline r$ is defined by the congruence
$$
\overline r r\equiv 1  \pmod \ell , \qquad 1 \le \overline r \le \ell-1.
$$ 
\end{lem}

Let, as usual, $\pi(X)$ denote the number of primes $\ell \le X$.
Then we have the following lower bound, due to Baker, Harman and Pintz~\cite{BaHaPi} 
for primes in short intervals, see also~\cite[Section~10.5]{IwKow} for several
related results. 

\begin{lem}
\label{lem:Primes}
There in an absolute constant $c> 0$ such that for $Y \ge X^{0.525}$ 
we have $\pi(X+Y)-\pi(X) \ge c Y/\log Y$.
\end{lem}

Finally we need some algorithmic results. 
We measure the complexity of our algorithms in the
so-called {\it RAM model of computation\/}, 
see~\cite{vzGG} for a background.

Since we are mostly interested in theoretic estimates,
we always assume that  fast arithmetic of long integers is 
used and in particular any arithmetic operation on $n$-bit 
integers can be performed in time $n^{1+o(1)}$ as $n\to\infty$,
see~\cite[Theorem~8.23]{vzGG}.

Besides we recall that for any prime $p$ and 
integers $a$, $b$ and $k$, all at most $n$-bits long, we can 
\begin{itemize}
\item compute the residue $a^k \pmod p$ in time $n^{2+o(1)}$, see~\cite[Section~4.3]{vzGG}.
\item compute the Legendre symbol $(a/p)$ in time $n^{1+o(1)}$, 
see~\cite{BrZim};
\item find an integer solution $(u,v)$ of the linear equation $au-bv=1$
in time $n^{1+o(1)}$,  see~\cite{Moll}.
\end{itemize}

We use these estimates throughout Sections~\ref{sec:N exact}
and~\ref{sec:N approx}.

\section{Bounding $N_{0}(p)$}
\label{sec:bounding}

The upper bound for $N_0(p)$ is immediate from 
Lemma~\ref{lem:HB}.

\begin{theorem}\label{thm:N0}
We have
$$
N_0(p) = O(p^{2/3}). 
$$
\end{theorem}

\begin{proof}
It is easy to verify that 
\begin{equation}
\label{eq:MultProp}
q_p(uv) \equiv q_p(u) + q_p(v) \pmod p
\end{equation}
for any integers $u$ and $v$ with $\gcd(uv,p) = 1$,
see, for example,~\cite[Equation~(3)]{ErnMet}.

Therefore 
\begin{equation}
\label{eq:qo expan}
\begin{split}
q_p\(s^s/(s+1)^{s+1}\) & \equiv sq_p(s) - (s+1)q_p(s+1)\\
& \equiv
\frac{s^{p} - (s+1)^{p} + 1}{p} \pmod p.
\end{split}
\end{equation}

It is shown by Heath-Brown~\cite[Section~1]{H-B} 
that for $1 \le s \le p-2$
$$
\frac{s^{p} - (s+1)^{p} + 1}{p}\equiv f(s+1) \pmod p,
$$
where $f(u)$ is as in Section~\ref{sec:prep}.
Applying Lemma~\ref{lem:HB} we conclude the proof.
\end{proof} 

Therefore the curve $Y^p = X^s(1-X)$ is tame for at 
most $O(p^{2/3})$ positive integers $s \le p-1$. 
 
It has been shown in the proof of~\cite[Lemma~6.1]{McCTz}
that for $p\equiv 1 \pmod 3$ we have $\vartheta_{p,s} = 0$ 
for both roots $s$ of the congruence $s^2 + s + 1 \equiv 0 \pmod p$. 
So $N_0(p)\ge 2$ for $p\equiv 1 \pmod 3$. We are not 
aware of any other lower bounds.

\section{Computing  $N_{\vartheta}(p)$}
\label{sec:N exact}

We start with an observation that the known algorithmic results 
presented in in Section~\ref{sec:prep} imply that 
the  values of $N_{0,\pm 1}(p)$ can be computed 
directly from the definition  in time and space
$$
T = p (\log p)^{2+o(1)} \mand S = O(\log p),
$$
respectively.   Indeed, this follows instantly from the 
congruence
$$
q_p\(s^s/(s+1)^{s+1}\)  \equiv sq_p(s) - (s+1)q_p(s+1)   \pmod p
$$
that is based on~\eqref{eq:MultProp} and which 
we have used in the proof of Theorem~\ref{thm:N0}.

If memory is not of concern, we can
simply compute the table of the values of $q_p(s)$,  
$s =1, \ldots, p-2$,  in $O(1)$ arithmetic operations 
modulo $p$
per value,  see~\cite[Theorem~7]{OstShp}.  After this, using fast arithmetic 
for the Legendre symbol, we can compute $N_{0,\pm 1}(p)$,
in time and space
$$
T = p (\log p)^{1+o(1)} \mand S = O(p\log p),
$$
respectively.

Furthermore, using~\cite[Algorithm~8]{OstShp} one can have some
trade-off between the space complexity and running time of
computing $N_\vartheta(p)$.  Indeed, for any parameter  $Z\ge 2$, we can 
evaluate in time  $pZ^{-1}(\log p)^{1+o(1)}$ a certain table 
(of size $O\(pZ^{-1}\log p\)$) such that after this 
for each $s =1, \ldots, p-2$ we can compute  
$\vartheta_{p,s}$ in time $(\log p)^{1+o(1)} \log Z$,
see~\cite[Theorem~9]{OstShp}.
Thus for any $Z\ge 2$,  we can compute $N_{0,\pm 1}(p)$ in time and space
$$
T = p(\log p)^{1+o(1)} \log Z \mand S =O\(pZ^{-1}\log p\),
$$
respectively.   Clearly $Z=p$ corresponds to the above trivial 
algorithm. However, taking $Z = \exp\(\sqrt{\log p}\)$ we see that we 
can compute $N_{0,\pm 1}(p)$ in 
time $p (\log p)^{3/2+o(1)}$ and space $p \exp\(-(1+o(1))\sqrt{\log p}\)$.
 
\section{Approximating $N_{\vartheta}(p)$}
\label{sec:N approx}

We now design  Quasi-Monte Carlo type algorithms that evaluate 
$\vartheta_{p,s}$ on a sequence of $s$ that is asymptotically uniformly 
distributed in the interval $[1, p-2]$ and have much 
more modest space requirements that the algorithms of Section~\ref{sec:N exact}.

Let 
$$U = \rf{p^{1/2}}
\mand
\Delta=\rf{p^{3/8} \log p}.
$$
We now precompute and store the table $\cQ$ of values $q_p(w)$, $1\le w\le U$, which 
can be done in 
in time and space
\begin{equation}
\label{eq: preproc}
T = p^{1/2} (\log p)^{2+o(1)} \mand S = O(p^{1/2} \log p),
\end{equation}
respectively. This is the cost of preprocessing.

Let  $\cL$ be the set of primes $\ell \in [U-\Delta, U]$
and let $\cR$ be the set of integers $r\in  [U-3\Delta, U-2\Delta]$. 
We now recall the complexity bounds of Section~\ref{sec:prep}
and proceed as follows:

\begin{algorithm}[Using Linear Equations] \qquad   
  \label{alg:LE}
  
\begin{description}

\item[Step 1:] Select  at random a prime $\ell \in \cL$ and an integer $r\in \cR$.

\item[Step 2:] Find positive integers $u < r$ and $v < \ell$ with $rv - \ell u = 1$ 
in time $(\log p)^{1+o(1)}$.

\item[Step 3:]  Set $s = \ell u$.

\item[Step 4:]  Using the precomputed table $\cQ$ and applying~\eqref{eq:MultProp}, 
we compute 
$$
q_p\(s^s/(s+1)^{s+1}\)  \equiv sq_p(\ell)q_p(u) - (s+1)q_p(r)q_p(v)   \pmod p,
$$
and then $\vartheta_{p,s}$, in time $(\log p)^{1+o(1)}$.
\end{description}
\end{algorithm}

Thus after the preprocessing with the cost given by~\eqref{eq: preproc}
we compute $\vartheta_{p,s}$ for every $s \in \cS$ in essentially 
linear time, where $\cS$ is the set of $s$ generated at Step~2 
of the above algorithm. That is,
$$
\cS = \{s = \ell u~:~ rv - \ell u = 1, \ \ell \in \cL, \ r\in \cR, \ 0 < u < r\}.
$$
We now show that indeed the above algorithm samples $s$ in a
reasonably uniform fashion.

\begin{theorem}\label{thm:Discr S}
The discrepancy of the sequence $(s/p)_{s \in \cS}$
is $ O\(p^{-1/8} \log p \)$.
\end{theorem}

\begin{proof} First of all we note that Lemma~\ref{lem:Primes}
applies to the set $\cL$ so it is not empty and contains at 
least $c_0\Delta/\log p$ primes, for some absolute 
constant $c_0>0$. 
 Furthermore, distinct 
pairs $(\ell, r) \in \cL\times \cR$ lead to distinct products
$s = \ell u$.

 Clearly for $\ell \in \cL$ and  $r\in \cR$ we have 
$$
\ell r  = U^2 + O(U\Delta) = p + O(U\Delta).
$$
Therefore, 
$$
\frac{s}{p}   = \frac{\ell u}{p} = \frac{rv}{p} - \frac{1}{p} = \frac{rv}{\ell r +  O(U\Delta)} - \frac{1}{p} = 
 \frac{v}{\ell} + O\(\frac{U\Delta}{p}\) .
$$

Using Lemmas~\ref{lem:DiscrStable} and~\ref{lem:Invers},
we see that the discrepancy of the sequence of fractions $s/p$
with  $s \in \cS$ corresponding to a given value of $\ell$ is
$$
O\(\Delta^{-1} U^{1/2}(\log U)^{2} + U\Delta p^{-1} \)
= O\(p^{-1/8} \log p \).
$$ 
Obviously the discrepancy of the entire sequence $(s/p)_{s \in \cS}$ also satisfies the same
bound. 
\end{proof} 

Unfortunately the proof of  Theorem~\ref{thm:Discr S} 
takes no advantage of averaging over $\ell$. So, this 
certainly can be a way to obtain a further improvement. 
In particular, it is possible that the switching over argument
of Fouvry~\cite{Fouv} maybe of help here. 

The above approach is based on constructing 
suitable values of $s$ from solutions of linear equations 
with the coefficients $\ell$ and $r$ running independently through some 
prescribed sets. There are various possible modifications that may lead to algorithmic advantages, and there are a variety of results~\cite{DiSi,Dolg,Fuji,Rieg,Shp0} that can be used to prove
the analogues of Theorem~\ref{thm:Discr S}
about the uniformity of distribution of the corresponding values of $s$. 

Clearly Algorithm~\ref{alg:LE} makes sense only if one intends to compute $\vartheta_{p,s}$ 
for at least $p^{1/2} \log p$ values of $s$.

\section{Relations between $\vartheta_{p,s}$}
\label{sec:relations}

Define permutations $F, G : \Sp \to \Sp$ by
 $$ F(s)  \equiv  - 1 - s \pmod p, \qquad G(s)  \equiv  1/s \pmod p. $$
One checks that $F^2 = G^2 = (FG)^3 = \text{id}_{\Sp}$, so the group $H$ generated by $F$ and $G$ contains at most six distinct permutations, namely $1$, $F$, $G$, $FG$, $GF$, and $FGF$ ($=GFG$). Furthermore, $F$, $G$ and $FGF$ each have a single fixed point (respectively $-1/2$, $1$, and $-2$ modulo $p$), and the fixed points of $FG$ and $GF$ are the roots of $x^2 + x + 1 = 0 \pmod p$ (if they exist).

From these facts one can easily determine the number of orbits of $\Sp$ under $H$, as follows. Assume that $p \geq 11$, so that $-1/2, 1, -2$, and the roots of $x^2 + x + 1 = 0 \pmod p$, are all distinct modulo $p$. First suppose that $p = 1 \pmod 3$. Let $s_0$ and $s_1$ be the roots of $x^2 + x + 1 = 0 \pmod p$. Note that $F$ and $G$ interchange $s_0$ and $s_1$. Thus there are precisely $(p+5)/6$ orbits, namely $\{-1/2, 1, -2\}$, $\{s_0, s_1\}$, and $(p-7)/6$ orbits of order 6. Now suppose that $p = 2 \pmod 3$. Then $x^2 + x + 1 = 0 \pmod p$ has no roots, and we obtain $(p+1)/6$ orbits, namely $\{-1/2, 1, -2\}$ and $(p-5)/6$ orbits of order $6$.

In all cases, we see that there are $p/6 + O(1)$ orbits. Next we show that $\vartheta_{p,s}$ is \emph{constant} on each orbit.

\begin{theorem}
We have $\vartheta_{p, F(s)} = \vartheta_{p, s}$ and $\vartheta_{p, G(s)} = \vartheta_{p, s}$.
\end{theorem}

\begin{proof}
We will use the fact that for any $u, v \in \Z$, with $u \neq 0 \pmod p$,
\begin{equation}
\label{eq:mod p}
\begin{split}
 q_p(u + vp) &  \equiv  \frac{(u+vp)^{p-1} - 1}p \\
              &  \equiv  \frac{u^{p-1} + (p-1)u^{p-2} vp - 1}{p} \\
              &  \equiv  q_p(u) - \frac{v}{u} \pmod p, 
\end{split}
\end{equation}
see also~\cite[Equation~(2)]{ErnMet}. 
For the first relation,
\begin{align*}
 \vartheta_{p, F(s)}
   & = \left(\frac{2(-s-1)(-s) \((-s-1) q_p(p-1-s) - (-s) q_p(p-s)\)}p\right) \\
   & = \left(\frac{2s(s+1)\(s q_p(s - p) - (s+1) q_p(s + 1 - p) \)}p\right).
\end{align*}
The result then follows from 
$$q_p(s-p)  \equiv  q_p(s) - \frac1s \pmod p
$$
and
$$
q_p(s+1-p)  \equiv q_p(s+1) - \frac1{s+1} \pmod p.
$$ 
For the second relation,
\begin{align*}
 \vartheta_{p, G(s)}
   & = \left(\frac{2 s^{-1} (s^{-1} + 1)\(\frac1s q_p(G(s)) - (s^{-1} + 1) q_p(G(s) + 1)\)}p\right) \\
   & = \left(\frac{2 s(s+1)\(q_p(G(s)) - (s+1) q_p(G(s) + 1)\)}p\right)
\end{align*}
since $(s^4/p) = 1$. Now let $s G(s) = 1 + kp$ for some $k \in \Z$. Then
 $$ q_p(G(s)) = q_p(s G(s)) - q_p(s) = q_p(1) - k - q_p(s) = -k - q_p(s) \pmod p $$
and
 $$ q_p(G(s) + 1)  \equiv  q_p(s G(s) + s) - q_p(s) = q_p(s+1) - \frac{k}{s+1} - q_p(s) \pmod p. $$
Therefore
 $$ q_p(G(s)) - (s+1)q_p(G(s) + 1)  \equiv  s q_p(s) - (s+1) q_p(s+1) \pmod p, $$
and $\vartheta_{p,G(s)} = \vartheta_{p,s}$ as desired.
\end{proof}
A natural question is whether there is some geometric explanation for the above relations. For example, are there maps between $\cF_s$, $\cF_{F(s)}$ and $\cF_{G(s)}$ that force them to have the same reduction type?

\section{Numerical Results} 

We have computed $N_{0, \pm 1}(p)$ for all $p < 10^7$ using a C implementation of a fairly naive algorithm.

Table~\ref{tab:moments} gives a statistical summary of the distribution of $N_1(p)$. The data for $N_{-1}(p)$ is similar and is not shown. The table has been constructed in the following way. Following the results of the previous section, $N_1(p)$ should behave like a sum of $p/6$ independent random variables that take the values $6$ and $0$ with equal probability. Each such variable has mean $3$ and variance $9$, so under this assumption we expect $N_1(p)$ to have mean $3(p/6) = p/2$ and variance $9(p/6) = 3p/2$. We treat each prime $p$ as an `observation' of $N_1(p)$, and define a normalised random variable
 $$ X = \frac{N_1(p) - p/2}{\sqrt{3p/2}}. $$
Table~\ref{tab:moments} shows the moments of $X$, compared to the moments of the standard normal distribution. The closeness of the fit strongly supports the assumptions of our model.

\begin{table}\centering
\caption{Moments of normalised $N_1(p)$, for $3 \leq p < 10^7$}
\begin{tabular}{lrr}
\toprule
$k$ & $E(X^k)$ & $E(N^k)$ \\
\midrule
1 &  $-0.00085$ & 0 \\
2 &   $0.99979$ & 1 \\
3 &   $0.00051$ & 0 \\
4 &   $3.00059$ & 3 \\
5 &   $0.00403$ & 0 \\
6 &  $14.92162$ & 15 \\
7 &  $-0.07897$ & 0 \\
8 & $102.90932$ & 105 \\
\bottomrule
\end{tabular}
\label{tab:moments}
\end{table}

Table \ref{tab:zero} summarises the behaviour of $N_0(p)$. If we assume that $\vartheta_{p,s}$ takes the value $0$ with probability $1/p$ on each of the $p/6 + O(1)$ orbits, then we expect $N_0(p)/6$ to behave like a Poisson random variable with mean $1/6$. In Table \ref{tab:zero}, the column $T_2(k)$ counts the number of primes $p = 2 \pmod 3$ such that $N_0(p) = 6k$; it closely matches the last column, which shows the value predicted by the Poisson model. For $p = 1 \pmod 3$ we must modify this slightly, because we know that automatically $\vartheta_{p,s} = 0$ when $s$ is one of the roots of $s^2 + s + 1 = 0 \pmod p$ (see Section \ref{sec:bounding}). This effectively increases $N_0(p)$ by two. In the table, we correspondingly define $T_1(k)$ to be the number of primes $p = 1 \pmod 3$ such that $N_0(p) = 6k + 2$. Again this closely matches the Poisson model.

Finally, one computes that 
$$
\vartheta_{p,1} = \(\frac{-2q_p(2)}p\)
$$ 
By definition, the latter is zero modulo $p$ if and only if $p$ is a \emph{Wieferich prime}. There are only two known Wieferich primes, namely $1093$ and $3511$. For these two primes, $\vartheta_{p,s}$ is zero on the orbit $\{1, -1/2, -2\}$. In fact $N_0(3511) = 5$ and $N_0(1093) = 17$ .

\begin{table}\centering
\caption{Frequency table for $N_0(p)$, for $5 \leq p < 10^7$, $p \neq 1093, 3511$}
\begin{tabular}{lrrr}
\toprule
$k$ & $T_1(k)$ & $T_2(k)$ & Poisson prediction \\
\midrule
0 & 281486 & 281127 & 281277 \\
1 & 46619 & 47088 & 46879 \\
2 & 3860 & 3923 & 3906 \\
3 & 217 & 231 & 217.03 \\
4 & 10 & 14 & 9.043 \\
\bottomrule
\end{tabular}
\label{tab:zero}
\end{table}

\section{Comments}

We remark that in the range of our calculations of $N_{0,\pm 1}(p)$,  none of the 
asymptotically faster algorithms of Sections~\ref{sec:N exact} and~\ref{sec:N approx}
were used. We however believe that these algorithms are not only of theoretic interest
and can become more practically useful for large values of $p$. 

We note that the definition of  $\vartheta_{p,s}$ makes sense for any integer $s$ with 
$\gcd(s(s+1),p)=1$ and then it obviously becomes a periodic function of $s$ 
with period $p^2$. Furthermore, a more careful analysis shows that it is 
periodic with period $p$. Indeed, using~\eqref{eq:mod p}, we see that
$$
(u+p)q_p(u+p) \equiv uq_p(u+p) \equiv uq_p(u) -1.
$$
Now, recalling~\eqref{eq:qo expan}, we derive 
\begin{equation*}
\begin{split}
q_p& \((s+p)^{s+p}/(s+p+1)^{s+p+1}\)  \\
& \qquad\equiv (s+p)q_p (s+p) - (s+p+1)q_p(s+p+1)\\
& \qquad\equiv sq_p(s) - (s+1)q_p(s+1) \equiv q_p\(s^s/(s+1)^{s+1}\)  \pmod p.
\end{split}
\end{equation*}
This explains why it is enough to study the values of $\vartheta_{p,s}$ only for $s \in \cS_p$.
This can also be used in  numerical tests as one can use any values of $s$ for which the relevant values of $q_p$ are easy to are easy to compute. 

For instance, assume that we want to compute $N_{0,\pm 1}(p)$ for 
many  primes $p \le X$. 
Then we can try to find a large set $\cS$ of  $s \in [1,X]$ (together with their factorisations) such that 
$s$ and $s+1$ are $Y$-smooth (that is, whose prime factors are all less than $Y$).
Then, for each $p$, 
we compute and store the values of $q_p(\ell)$ for all primes $\ell \le Y$.
After this $\vartheta_{p,s}$, $s \in \cS$, can be computed very rapidly. 
Certainly for this approach to work, 
we need to know for what $Y$ there many such integers $s$ 
and how they are  distributed in $[1,X]$.  However, we remark that the
questions of counting and generating smooth pairs $(s,s+1)$ is apparently 
very difficult, see~\cite[Section~6]{LagSound1}.

We can   expand further the set 
of potentially ``friendly'' test points by considering triples 
$(u,v,w)$ of $Y$-smooth positive integers $u,v,w \le X$. After this 
we compute  $t \equiv u/w \pmod {p^2}$ and then compute 
$s \in \cS_p$ with $s \equiv t \pmod p$. We now have $\vartheta_{p,s} = \vartheta_{p,t}$,
while $t$ is a ``friendly'' value.
Unfortunately counting and generating such triples $(u,v,w)$ is 
still a difficult problem, see~\cite{LagSound1,LagSound2}.

\section*{Acknowledgments}

The authors are very grateful to \'Etienne Fouvry for clarifying discussions 
of the issue of existence and generating neighbouring smooth numbers.

The research of  D.~H. was partially supported
by  Australian Research Council Grant DE120101293
and that of  I.~E.~S. by Australian Research Council Grant  DP1092835.


\begin{thebibliography}{99}


\bibitem{BaHaPi}
R. C. Baker and G. Harman and J. Pintz,
`The difference between consecutive primes. II',
{\it Proc. Lond. Math. Soc.\/},  {\bf 83} (2001), 532--562.

\bibitem{BrZim}
R. Brent and P. Zimmermann, 
`An $O(M(n) \log n)$ algorithm for the Jacobi symbol',  
{\it Proc. 9th Algorithmic Number Theory Symposium (ANTS-IX), Nancy, 2010\/}
Lect. Notes. Comp. Sci., v.~6197,  Springer-Verlag,  83--95.

\bibitem{Chang}
M.-C. Chang, `Short character sums with Fermat quotients', 
{\it Acta Arith.\/}, {\bf 152} (2012), 23--38.

\bibitem{DiSi} 
 E. I. Dinaburg and Y. G. Sinai, 
`The statistics of the solutions of the integer equation $ax-by=\pm 1$',
{\it Funktsional. Anal. i Prilozhen.
 (Transl. as Funct. Anal. Appl.)\/},
{\bf 24}(3) (1990),   1--8 (in Russian). 


\bibitem{Dolg} D. Dolgopyat, `On the distribution of the minimal 
solution of a linear Diophantine
equation with random coefficients', 
{\it Funktsional. Anal. i Prilozhen. 
(Transl. as Funct. Anal. Appl.)\/}, 
{\bf 28}(3) (1994), 22--34 (in Russian). 

\bibitem{DrTi} M.  Drmota and R.  Tichy,
\emph{Sequences, discrepancies and applications\/},
Springer-Verlag, Berlin, 1997.

\bibitem{ErnMet} R. Ernvall and T. Mets{\"a}nkyl{\"a},
`On the $p$-divisibility of Fermat quotients',
{\it Math. Comp.\/}, {\bf  66} (1997),  1353--1365.

\bibitem{Fouv} {\'E}. Fouvry,
{\em `Sur le probl{\'e}me des diviseurs de Titchmarsh'},
J. Reine Angew. Math.  {\bf 357} (1985),  51--76.

\bibitem{Fuji}
A. Fujii, `On a problem of Dinaburg and Sinai',
{\it Proc. Japan Acad. Sci., Ser.~A\/}, {\bf 68} (1992), 198--203.

\bibitem{vzGG}
J. von zur Gathen and J. Gerhard, {\it Modern computer algebra\/},
Cambridge University Press, Cambridge, 2003.
 
\bibitem{H-B} R. Heath-Brown, `An estimate for Heilbronn's exponential sum', {\it Analytic Number Theory:
Proc. Conf.  in honor of Heini Halberstam\/}, 
Birkh{\"a}user, Boston, 1996, 451--463. 
%

\bibitem{IwKow} H. Iwaniec and E. Kowalski,  
{\it Analytic number theory\/},  Amer. Math. 
Soc., Providence RI, 2004.

\bibitem{LagSound1} J. C. Lagarias and K. Soundararajan, `Smooth solutions to the 
$abc$ equation: the $xyz$ conjecture',
{\it J. Th{\'e}or. Nombres Bordeaux\/}, {\bf  23} (2011),  209--234.   

\bibitem{LagSound2} J. C. Lagarias and K. Soundararajan, 
`Counting smooth solutions to the equation $A+B=C$',
{\it Proc. Lond. Math. Soc.\/}, {\bf 104} (2012), 770--798. 


\bibitem{LeMcC} B. Levitt and W. McCallum, 
`Yet more elements in the Shafarevich-Tate group of the
Jacobian of a Fermat curve', {\it Computational Arithmetic Geometry: 
AMS Special Session, San Francisco, CA, USA, April 29--30, 2006\/},
Contem. Math., vol.~463, Amer. Math. Soc., Providence, RI, 2008, 83--90.
 
\bibitem{McC1} W. McCallum, 
`The degenerate fibre of the Fermat curve', 
{\it Number theory related to Fermat's last theorem (Cambridge, MA, 1981)\/},
Progr. Math., vol.~26, Birkh{\"a}user, Boston, MA, 1982,   57--70.
 
\bibitem{McC2} W. McCallum, `On the Shafarevich-Tate group of the Jacobian of a 
quotient of the Fermat curve', {\it Invent. Math.\/}, {\bf 93} (1988),  637--666.   

\bibitem{McCTz} W. McCallum and P. Tzermias, 
`On Shafarevich-Tate groups and the arithmetic of Fermat curves', 
{\it Number theory and algebraic geometry\/}, 
London Math. Soc. Lecture Note Ser., vol.~303, 
Cambridge Univ. Press, Cambridge, 2003, 203--226.
 
\bibitem{Mit}  D. A. Mit'kin, `On estimating the number of roots of some congruences by the Stepanov method', 
{\it Matem. Zametki\/}, {\bf 51}(6) (1992), 52--58 (in Russian);
Translated {\it Math. Notes\/}, {\bf 51} (1992), 565--570. 

\bibitem{Moll} N.  M\"oller, `On Sch\"onhage�s
algorithm and subquadratic integer GCD computation', {\it Math.
Comp.\/}, {\bf 77} (2008), 589--607.

 
\bibitem{OstShp}
A. Ostafe and I.~E.~Shparlinski,  
`Pseudorandomness and dynamics of Fermat quotients',
{\it SIAM J. Discr. Math.\/},   {\bf 25} (2011),  50--71.

\bibitem{Rieg} G. J. Rieger, 
`{\" U}ber die Gleichung $ad-bc=1$ und Gleichverteilung',
{\it Math. Nachr.\/}, {\bf  162} (1993), 139--143. 


\bibitem{Shp0} I. E. Shparlinski, 
`On the distribution of solutions to 
linear equations',  {\it Glasnik Math.\/},  {\bf 44} (2009),
7--10.
 

\bibitem{Shp1} 
I. E. Shparlinski,
`Character sums with Fermat quotients',  
{\it Quart. J. Math.\/}, {\bf 62} (2011), 1031--1043.



\bibitem{Shp2} 
I. E. Shparlinski,
`Bounds of multiplicative 
character sums with Fermat quotients of primes',  
{\it Bull. Aust. Math. Soc.\/},  {\bf 83} (2011), 456--462.

\bibitem{Shp3} I. E. Shparlinski, 
`Fermat quotients: Exponential sums, value set and primitive roots', 
{\it Bull. Lond. Math. Soc.\/}, {\bf 43} (2011), 1228--1238.

\bibitem{Shp4} I. E. Shparlinski, `Modular hyperbolas', 
{\it  Japanese J. Math.\/}, {\bf 7} (2012), 235--294.

%

\end{thebibliography}
\end{document}